\documentclass[12pt,twoside]{amsart}
\usepackage{amssymb}
\usepackage{amscd}
\usepackage{amsmath}
\usepackage{xypic}

\title[Weak Fano manifolds II]
{On images of weak Fano manifolds II}
\author{Osamu Fujino} 
\date{2012/1/5, version 1.20}
\address{Department of Mathematics, Faculty of Science, 
Kyoto University, Kyoto 606-8502, Japan}
\email{fujino@math.kyoto-u.ac.jp}

\address{Graduate School of Mathematical Sciences, 
The University of Tokyo, 3-8-1 Komaba, 
Meguro, Tokyo, 153-8914 Japan.}
\email{gongyo@ms.u-tokyo.ac.jp}
\author{Yoshinori Gongyo} 
\newcommand{\Supp}[0]{{\operatorname{Supp}}}

\newtheorem{thm}{Theorem}[section]
\newtheorem{lem}[thm]{Lemma}
\newtheorem{cor}[thm]{Corollary}

\newtheorem*{claim}{Claim}
\newtheorem{conj}[thm]{Conjecture}

\theoremstyle{definition}

\newtheorem{rem}[thm]{Remark}
\newtheorem*{ack}{Acknowledgments}       

\newtheorem{say}[thm]{}
\begin{document}

\maketitle 

\begin{abstract}
We consider a smooth projective surjective morphism between smooth complex projective varieties. 
We give a Hodge theoretic proof of the following well-known fact:~If 
the anti-canonical divisor of the source space is nef, 
then so is the anti-canonical divisor of the target space. 
We do not use mod $p$ reduction arguments. 
In addition, we make some supplementary comments on our paper:~On images of weak Fano manifolds. 
\end{abstract} 

\tableofcontents
\section{Introduction}

We will work over $\mathbb C$, the complex number field. 
The following theorem is the main result of this paper. It is a generalization of 
\cite[Corollary 3.15 (a)]{deb}. 

\begin{thm}[Main theorem]\label{main-thm}
Let $f:X\to Y$ be a smooth projective surjective morphism between smooth projective varieties. 
Let $D$ be an effective $\mathbb Q$-divisor on $X$ such that 
$(X, D)$ is log canonical, $\Supp D$ is a simple normal crossing 
divisor, 
and $\Supp D$ is relatively normal crossing over $Y$. 
Let $\Delta$ be a {\em{(}}not necessarily effective{\em{)}} $\mathbb Q$-divisor on $Y$. 
Assume that $-(K_X+D)-f^*\Delta$ is nef. 
Then $-K_Y-\Delta$ is nef. 
\end{thm}

By putting $D=0$ and $\Delta=0$ in Theorem \ref{main-thm}, we obtain the following corollary. 

\begin{cor}
Let $f:X\to Y$ be a smooth projective surjective morphism between smooth projective varieties. 
Assume that $-K_X$ is nef. 
Then $-K_Y$ is nef. 
\end{cor}

By putting $D=0$ and assuming that $\Delta$ is a small ample $\mathbb Q$-divisor, 
we can recover \cite[Corollary 2.9]{kmm} by Theorem \ref{main-thm}. Note that 
Theorem \ref{main-thm} is also a generalization of \cite[Theorem 4.8]{fg}. 

\begin{cor}[{cf.~\cite[Corollary 2.9]{kmm}}] 
Let $f:X\to Y$ be a smooth projective surjective morphism between smooth projective varieties. 
Assume that $-K_X$ is ample. 
Then $-K_Y$ is ample. 
\end{cor}

Note that Conjecture 1.3 in \cite{fg} 
is still open. The reader can find some affirmative 
results on Conjecture \ref{conj13} in \cite[Section 4]{fg}. 

\begin{conj}[Semi-ampleness conjecture]\label{conj13} 
Let $f:X\to Y$ be a smooth projective surjective morphism between smooth projective varieties. 
Assume that $-K_X$ is semi-ample. 
Then $-K_Y$ is semi-ample. 
\end{conj}

In this paper, we give a proof of Theorem \ref{main-thm} 
without mod $p$ reduction arguments. 
Our proof is Hodge theoretic. 
We use a generalization of Viehweg's weak positivity theorem following \cite{cz}. 
In our previous paper \cite{fg}, we just use 
Kawamata's positivity theorem. 
We note that Theorem \ref{main-thm} is better than \cite[Theorem 4.1]{fg} 
(see Theorem \ref{thm23} below). 
We also note that Kawamata's positivity theorem (cf.~\cite[Theorem 2.2]{fg}) and 
Viehweg's weak positivity theorem (and its generalization 
in \cite[Theorem 4.13]{campana}) are obtained by Fujita--Kawamata's 
semi-positivity theorem, which follows from 
the theory of the variation of (mixed) Hodge structure. 
We recommend the readers to compare the proof of Theorem \ref{main-thm} 
with the arguments in \cite[Section 4]{fg}. 
By the Lefschetz principle, all the results in this paper 
hold over any algebraically closed field $k$ of characteristic zero. 
We do not discuss the case when the characteristic of the base field is positive. 

\begin{ack}
The first author was partially supported by the Grant-in-Aid for Young Scientists 
(A) $\sharp$20684001 from JSPS. 
The second author was partially supported by 
the Research Fellowships of the Japan Society for the Promotion of Science for Young Scientists.
The authors would like to thank Professor Sebastien Boucksom for informing them 
of Berndtsson's results \cite{bo}. 
They also would like to 
thank the Erwin Schr\"odinger International Institute for Mathematical 
Physics in Vienna for its hospitality.
\end{ack}

\section{Proof of the main theorem}
In this section, we prove Theorem \ref{main-thm}. 
We closely follow the arguments of \cite{cz}. 

\begin{lem}\label{lem-cz} 
Let $f:Z\to C$ be a projective surjective morphism from a $(d+1)$-dimensional 
smooth projective variety $Z$ to a smooth projective curve $C$. 
Let $B$ be an ample Cartier divisor on $Z$ such that $R^if_*\mathcal O_Z(kB)=0$ 
for every $i>0$ and $k\geq 1$. 
Let $H$ be a very ample Cartier divisor on $C$ such that 
$B^{d+1}<f^*(H-K_C)\cdot B^d$ and $B^{d+1}\leq f^*H\cdot B^d$. 
Then $$(f_*\mathcal O_Z(kB))^*\otimes \mathcal O_C(lH)$$ is 
generated by global sections for $l>k\geq 1$.  
\end{lem}
\begin{proof}
By the Grothendieck duality 
$$
R\mathcal H om(Rf_*\mathcal O_Z(kB), \omega_C^\bullet)\simeq Rf_*R\mathcal H om 
(\mathcal O_Z(kB), \omega_Z^\bullet), 
$$ 
we obtain $$(f_*\mathcal O_Z(kB))^*\simeq R^df_*\mathcal O_Z(K_{Z/C}-kB)$$ for $k\geq 1$ and 
$$R^if_*\mathcal O_Z(K_{Z/C}-kB)=0$$ for $k\geq 1$ and $i\ne d$. 
We note that $f_*\mathcal O_Z(kB)$ is locally free and $(f_*\mathcal O_Z(kB))^*$ is its dual locally free 
sheaf. 
Therefore, 
we have 
\begin{align*}
&H^1(C, (f_*\mathcal O_Z(kB))^*\otimes \mathcal O_C((l-1)H))\\
&\simeq H^1(C, R^df_*\mathcal O_Z(K_{Z/C}-kB)\otimes \mathcal O_C((l-1)H))\\
&\simeq H^{d+1}(Z, \mathcal O_Z(K_Z-f^*K_C-kB+f^*(l-1)H))
\end{align*} 
for $k\geq 1$. 
By the Serre duality, 
$$
H^{d+1}(Z, \mathcal O_Z(K_Z-f^*K_C-kB+f^*(l-1)H))
$$
is dual to 
$$
H^0(Z, \mathcal O_Z(kB+f^*K_C-f^*(l-1)H)). 
$$ 
On the other hand, by the assumptions 
$$
(kB+f^*K_C-f^*(l-1)H)\cdot B^d<0
$$ 
if $l-1\geq k$. 
Thus, we obtain 
$$
H^0(Z, \mathcal O_Z(kB+f^*K_C-f^*(l-1)H))=0 
$$ 
for $l>k$. 
This means that 
$$
H^1(C, (f_*\mathcal O_Z(kB))^*\otimes \mathcal O_C((l-1)H))=0
$$ 
for $k\geq 1$ and $l>k$. 
Therefore, 
$(f_*\mathcal O_Z(kB))^*\otimes \mathcal O_C(lH)$ is generated by 
global sections for $k\geq 1$ and $l>k$. 
\end{proof}

Let us start the proof of Theorem \ref{main-thm}. 

\begin{proof}[Proof of {\em{Theorem \ref{main-thm}}}] 
We prove the following claim. 
\begin{claim}
Let $\pi:C\to Y$ be a projective morphism from a smooth projective curve $C$ and 
let $L$ be an ample Cartier divisor 
on $C$. 
Then $(-\pi^*K_Y-\pi^*\Delta+2\varepsilon L)\cdot C\geq 0$ for any positive rational number $\varepsilon$. 
\end{claim} 
Let us start the proof of Claim. 
We fix an arbitrary positive rational number $\varepsilon$. 
We may assume that $\pi(C)$ is a curve, that is, 
$\pi$ is finite. 
We consider the following base change diagram 
$$
\begin{CD}
Z@>{p}>> X\\
@V{g}VV @VV{f}V\\
C@>>{\pi}> Y
\end{CD}
$$ 
where $Z=X\times _YC$. 
Then $g:Z\to C$ is smooth, $Z$ is smooth, $\Supp(p^*D)$ is relatively normal crossing over $C$, 
and $\Supp(p^*D)$ is a simple normal crossing divisor on $Z$. 
Let $A$ be a very ample Cartier divisor on $X$ and let $\delta$ 
be a small positive rational number such that $0<\delta\ll \varepsilon$. Since 
$-(K_X+D)-f^*\Delta+\delta A$ is ample, 
we can take a general effective $\mathbb Q$-divisor $F$ on $X$ such that 
$-(K_X+D)-f^*\Delta+\delta A\sim _{\mathbb Q}F$. 
Then we have 
$$
K_{X/Y}+D+F\sim _{\mathbb Q}\delta A-f^*K_Y-f^*\Delta. 
$$
By taking the base change, 
we obtain 
$$
K_{Z/C}+p^*D+p^*F\sim _{\mathbb Q} \delta p^*A-g^*\pi^*K_Y-g^*\pi^*\Delta. 
$$ 
Without loss of generality, 
we may assume that 
$\Supp (p^*D+p^*F)$ is a simple normal crossing divisor, 
$p^*D$ and $p^*F$ have no common irreducible components, and $(Z, p^*D+p^*F)$ is log canonical. 
Let $m$ be a sufficiently divisible positive integer such that 
$m\delta$ and $m\varepsilon$ are 
integers, $mp^*D$, $mp^*F$, and $m\Delta$ are Cartier divisors, and 
$$
m(K_{Z/C}+p^*D+p^*F)\sim m(\delta p^*A-g^*\pi^*K_Y-g^*\pi^*\Delta). 
$$ 
Note that $g:Z\to C$ is smooth, every irreducible component of $p^*D+p^*F$ is dominant onto $C$, and 
the coefficient of 
any irreducible component of $m(p^*D+p^*F)$ is a positive 
integer with $\leq m$. 
Therefore, we can apply the weak positivity theorem (cf.~\cite[Theorem 4.13]{campana}) and 
obtain that 
$$g_*\mathcal O_Z(m(K_{Z/C}+p^*D+p^*F))\simeq 
g_*\mathcal O_Z(m(\delta p^*A-g^*\pi^*K_Y-g^*\pi^*\Delta))$$ is 
weakly positive over some non-empty Zariski open set $U$ of 
$C$. For the basic properties of weakly positive sheaves, 
see, for example, \cite[Section 2.3]{viehweg}. 
Therefore, 
\begin{align*}
\mathcal E_1&:=
S^n(g_*\mathcal O_Z(m(\delta p^*A-g^*\pi^*K_Y-g^*\pi^*\Delta)))
\otimes \mathcal O_C(nm\varepsilon L)
\\
&\simeq S^n(g_*\mathcal O_Z(m\delta p^*A))\otimes 
\mathcal O_C(-nm\pi^*K_Y-nm\pi^*\Delta+nm\varepsilon L)
\end{align*} 
is generated by global sections over $U$ for every $n\gg 0$. On the other hand, by Lemma \ref{lem-cz}, 
if $m\delta\gg 0$, 
then we have that 
$$
\mathcal E_2:=\mathcal O_C(nm\varepsilon L)\otimes S^n((g_*\mathcal O_Z(m\delta p^*A))^*)
$$ 
is generated by global sections because $0<\delta\ll \varepsilon$ and $p^*A$ 
is ample on $Z$. 
We note that 
$$
\mathcal E_2\simeq S^n(\mathcal O_C(m\varepsilon L)\otimes (g_*\mathcal O_Z(m\delta p^*A))^*). 
$$ 
Thus there is a homomorphism 
$$
\alpha:\underset{\text{finite}}\bigoplus \mathcal O_C\to 
\mathcal E:=\mathcal E_1\otimes \mathcal E_2
$$ 
which is surjective over $U$. 
By using the non-trivial trace map 
$$
S^n(g_*\mathcal O_Z(m\delta p^*A))\otimes 
S^n((g_*\mathcal O_Z(m\delta p^*A))^*)\to \mathcal O_C, 
$$ 
we have a non-trivial homomorphism 
$$
\underset{\text{finite}}\bigoplus \mathcal O_C\overset{\alpha}{\longrightarrow} \mathcal E\overset{\beta}
{\longrightarrow} 
\mathcal O_C(-nm\pi^*K_Y-nm\pi^*\Delta+2nm\varepsilon L), 
$$ 
where $\beta$ is induced by the above trace map. 
We note that $g_*\mathcal O_Z(m\delta p^*A)$ is locally free and 
$$
S^n((g_*\mathcal O_Z(m\delta p^*A))^*)\simeq (S^n(g_*\mathcal O_Z(m\delta p^*A)))^*. 
$$ 
Thus we obtain $$(-nm\pi^*K_Y-nm\pi^*\Delta+2nm\varepsilon L)\cdot C=
nm(-\pi^*K_Y-\pi^*\Delta+2\varepsilon L)\cdot C\geq 0.$$ We finish the proof of Claim. 

Since $\varepsilon $ is an arbitrary small positive rational number, 
we obtain $\pi^*(-K_Y-\Delta)\cdot C \geq 0$. 
This means that $-K_Y-\Delta$ is nef on $Y$. 
\end{proof}

\begin{rem}
In Theorem \ref{main-thm}, 
if $-(K_X+D)$ is semi-ample, 
then we can simply prove that $-K_Y$ is nef as follows. 
First, by the Stein factorization, we may assume that $f$ has connected 
fibers (cf.~\cite[Lemma 2.4]{fg}). 
Next, in the proof of Theorem \ref{main-thm}, 
we can take $\delta=0$ and $\Delta=0$ when $-(K_X+D)$ is semi-ample. 
Then $$g_*\mathcal O_Z(m(K_{Z/C}+p^*D+p^*F))\simeq \mathcal O_C(-m\pi^*K_Y)$$ is weakly positive 
over some non-empty Zariski open set $U$ of $C$. 
This means that $-m\pi^*K_Y$ is pseudo-effective. 
Since $C$ is a smooth projective curve, $-\pi^*K_Y$ is nef. 
Therefore, $-K_Y$ is nef. In this case, we do not need Lemma \ref{lem-cz}. 
The proof given here is simpler than the 
proof of \cite[Theorem 4.1]{fg}. 
\end{rem}

We apologize the readers of \cite{fg} for misleading them on \cite[Theorem 4.1]{fg}. 
A Hodge theoretic proof of \cite[Theorem 4.1]{fg} 
is implicitly contained in Viehweg's theory of 
weak positivity (see, for example, \cite{viehweg}). Here we give a proof 
of \cite[Theorem 4.1]{fg} following Viehweg's arguments. 

\begin{thm}[{cf.~\cite[Theorem 4.1]{fg}}]\label{thm23} 
Let $f:X\to Y$ be a smooth projective surjective morphism between 
smooth projective varieties. If $-K_X$ is semi-ample, then $-K_Y$ is nef. 
\end{thm}
\begin{proof}
By the Stein factorization, we may assume that $f$ has connected fibers (cf.~\cite[Lemma 2.4]{fg}). 
Note that a locally free sheaf $\mathcal E$ on $Y$ is nef, equivalently, semi-positive in the sense of 
Fujita--Kawamata, if and only if $\mathcal E$ is weakly positive over $Y$ (see, for example, \cite[Proposition 
2.9 e)]{viehweg}). Since $f$ is smooth and $-K_X$ is semi-ample, 
$f_*\mathcal O_X(K_{X/Y}-K_X)$ is locally free and weakly positive over $Y$ 
(cf.~\cite[Proposition 2.43]{viehweg}). 
Therefore, 
we obtain that $\mathcal O_Y(-K_Y)\simeq f_*\mathcal O_X(K_{X/Y}-K_X)$ is nef. 
\end{proof}

Note that our Hodge theoretic proof of \cite[Theorem 4.1]{fg}, which depends on 
Kawamata's positivity theorem, is different from the 
proof given above and plays important roles in \cite[Remark 4.2]{fg}, which is related to Conjecture 
\ref{conj13}. 

\begin{say}[Analytic method]
Sebastien Boucksom pointed out that the following theorem, which is 
a special case of \cite[Theorem 1.2]{bo}, implies \cite[Theorem 4.1]{fg} and \cite[Corollary 2.9]{kmm}. 

\begin{thm}[{cf.~\cite[Theorem 1.2]{bo}}]
Let $f:X\to Y$ be a proper smooth morphism from a 
compact K\"ahler manifold $X$ to a compact complex manifold $Y$. 
If $-K_X$ is semi-positive {\em{(}}resp.~positive{\em{)}}, 
then $-K_Y$ is semi-positive {\em{(}}resp.~positive{\em{)}}. 
\end{thm}

The proof of \cite[Theorem 1.2]{bo} is analytic  and does not use mod $p$ reduction arguments. 
For the details, see \cite{bo}. 
\end{say}

\begin{say}[Varieties of Fano type] 
Let $X$ be a normal projective variety. 
If there is an effective $\mathbb Q$-divisor on $X$ such that 
$(X, \Delta)$ is klt and that $-(K_X+\Delta)$ is ample, then 
$X$ is said to be {\em{of Fano type}}. 

In \cite[Theorem 2.9]{ps} and \cite[Corollary 3.3]{fg}, 
the following statement was proved. 

Let $f:X\to Y$ be a proper surjective morphism between normal projective 
varieties with connected fibers. 
If $X$ is of Fano type, 
then so is $Y$. 

It is indispensable for the proof of the main theorem in 
\cite{fg} (cf.~\cite[Theorem 1.1]{fg}). 
The proofs in \cite{ps} and \cite{fg} need the theory of the variation of 
Hodge structure. 
It is because we use Ambro's canonical bundle formula 
or Kawamata's positivity theorem. 
In \cite{gost}, Okawa, Sannai, Takagi, and the second author 
give a new proof of the above result without using the theory of the variation of Hodge structure. It 
deeply depends on the minimal model theory and the theory of 
$F$-singularities. 
\end{say}

We close this paper with a remark on \cite{deb}. 
By modifying the proof of Theorem \ref{main-thm} suitably, 
we can generalize \cite[Corollary 3.14]{deb} without any difficulties. 
We leave the details as an exercise for the 
readers. 

\begin{cor}[{cf.~\cite[Corollary 3.14]{deb}}] 
Let $f:X\to Y$ be a projective surjective morphism from a smooth 
projective 
variety $X$ such that 
$Y$ is smooth in codimension one. 
Let $D$ be an effective $\mathbb Q$-divisor 
on $X$ such that $\Supp D^{\mathrm{hor}}$, 
where $D^{\mathrm{hor}}$ is the horizontal part of $D$, is a simple normal crossing divisor 
on $X$ and that $(X, D)$ is log canonical over the generic point of $Y$. 
Let $\Delta$ be a {\em{not necessarily effective}} 
$\mathbb Q$-Cartier $\mathbb Q$-divisor on $Y$. 
\begin{itemize}
\item[(a)] If $-(K_X+D)-f^*\Delta$ is nef, then 
$-K_Y-\Delta$ is generically nef. 
\item[(b)] If $-(K_X+D)-f^*\Delta$ is ample, then 
$-K_Y-\Delta$ is generically ample.  
\end{itemize} 
\end{cor}

\end{document}